\documentclass[a4paper,12pt]{article}

\usepackage{amsmath,amsfonts,amssymb,amscd,amsthm,eucal}

\usepackage[all]{xy}
\SelectTips{cm}{12}
\theoremstyle{plain} {
  \newtheorem{theorem}{Theorem}[subsection]
  \newtheorem{corollary}[theorem]{Corollary}

}

\theoremstyle{definition}
{
  \newtheorem{definition}[theorem]{Definition}
  \newtheorem{remark}[theorem]{Remark}
  \newtheorem{claim}[theorem]{Claim}
  \newtheorem{example}[theorem]{Example}
}
\renewcommand{\subsubsection}{\sssection\rm}

\newcommand{\id}{\mathrm{id}}
\newcommand{\Spec}{\mathrm{Spec}}
\newcommand{\Hom}{\mathrm{Hom}}

\newcommand{\Th}{\mathrm{Th}}

\newcommand{\SmOp}{\mathcal Sm\mathcal Op}
\newcommand{\Sm}{\mathcal Sm}

\newcommand{\Aff}{\mathbf {A}}
\newcommand{\MSS}{\mathbf{MSS}}
\newcommand{\Pro}{\mathbf {P}}
\newcommand{\Gr}{\mathrm{Gr}}

\newcommand{\K}{\mathcal K}

\newcommand{\ZZ}{\mathbb{Z}}
\let\minus=\smallsetminus
\let\tensor=\otimes
\newcommand{\colim}{\operatorname{colim}}
\newcommand{\SH}{\operatorname{SH}}
\newcommand{\cm}{\mathrm{cm}}

\newcommand \hra {\hookrightarrow }

\begin{document}

\title{On the relation of Voevodsky's algebraic cobordism to Quillen's $K$-theory}

\author{I.~Panin\footnote{Universit\"at Bielefeld, SFB 701, Bielefeld, Germany}
\footnote{Steklov Institute of Mathematics at St.~Petersburg, Russia}
\and K.~Pimenov\footnotemark[2]
\and O.~R{\"o}ndigs\footnote{Institut f\"ur Mathematik, Universit\"at Osnabr\"uck, Osnabr\"uck, Germany}}

\date{September 26, 2007}

\maketitle

\begin{abstract}
Quillen's algebraic $K$-theory is reconstructed via Voevodsky's
algebraic cobordism. More precisely,
for a ground field $k$
the algebraic cobordism $\Pro^1$-spectrum
$\mathrm{MGL}$ of Voevodsky
is considered as
a commutative $\Pro^1$-ring spectrum.
Setting
$\mathrm{MGL}^i = \oplus_{p-2q = i} \ \mathrm{MGL}^{p,q}$
we regard the bigraded theory
$\mathrm{MGL}^{p,q}$
as just a graded theory.
There is a unique ring morphism
$\phi\colon \mathrm{MGL}^0(k) \to \mathbb Z$
which sends the class
$[X]_{\mathrm{MGL}}$
of a smooth projective $k$-variety $X$ to the Euler
characteristic
$\chi(X, \mathcal O_X)$
of the structure sheaf
$\mathcal O_X$.
Our main result states that
there is a canonical grade preserving isomorphism of ring
cohomology theories 
\[\xymatrix{
{\varphi}\colon \mathrm{MGL}^{\ast}(X,U) 
\otimes_{\mathrm{MGL}^{0}(k)} \mathbb Z \ar[r]^-{\cong} &
\mathrm{K}^{TT}_{- *}(X,U)= \mathrm{K}^{\prime}_{- *}(X-U)}
\]
on the category
$\SmOp/k$
in the sense of
\cite{PSorcoh},
where
$K^{TT}_*$
is Thomason-Trobaugh $K$-theory and
$\mathrm{K}^{\prime}_*$
is Quillen's
$\mathrm{K}^{\prime}$-theory.
In particular, the left hand side is a ring cohomology theory.
Moreover both theories are oriented in the sense of
\cite{PSorcoh}
and
$\varphi$
respects the orientations.
The result is an algebraic version of a theorem due to Conner and Floyd.
That theorem reconstructs complex $K$-theory via complex cobordism
\cite{CF}.
\end{abstract}

\section{Introduction}

The isomorphism $\varphi$ is constructed in two steps. In the first step we
supply
an isomorphism
\begin{equation}\label{eq:2}
\bar \varphi\colon \mathrm{MGL}^{\ast}(X,U) \otimes_{\mathrm{MGL}^{0}(k)} \mathbb \mathrm{BGL}^{0}(k) \to
\mathrm{BGL}^{\ast}(X,U)
\end{equation}
where
$\mathrm{BGL}$
is the Voevodsky $K$-theory $\Pro^1$-spectrum representing 
$(2,1)$-periodic algebraic $K$-theory. To describe the second step recall that
$\mathrm{BGL}^{\ast}=\mathrm{K}^{TT}_{- *}[\beta,\beta^{-1}]$
for an element
$\beta \in \mathrm{BGL}^{2,1}(k) \subseteq \mathrm{BGL}^{0}(k)$.
Taking the quotient of both sides of the last isomorphism modulo
the ideal generated by the element
$(\beta + 1)$
we get the required isomorphism
$\bar {\bar \varphi}$.

Here are a few words on the construction of
$\bar \varphi$.
The $\Pro^1$-spectrum
$\mathrm{MGL}$
is defined as
$(S^0, \Th(\mathcal{T}(1)), \Th(\mathcal{T}(2)), \dots )$,
where
$\mathcal{T}(i)$
is the tautological vector bundle over the Grassmannian
$\mathrm{Gr}(i, \infty)$
and
$\Th(\mathcal{T}(i))$
is its Thom space.
There is a monoidal structure on
$\mathrm{MGL}$
described in
\cite{UniThm}.
The obvious morphism
$$
(\ast, \Th(\mathcal{T}(1)), \Th(\mathcal{T}(1)) \wedge \Pro^1, \dots ) \to
(S^0, \Th(\mathcal{T}(1)), \Th(\mathcal{T}(2)), \dots )
$$
defines an element
$th^{\mathrm{MGL}} \in \mathrm{MGL}^{2,1}(\Th(\mathcal{T}(1)))$,
which is a tautological orientation of
$\mathrm{MGL}$.
It turns out that the pair
$(\mathrm{MGL},th^{\mathrm{MGL}})$
is universal among all pairs
$(E, th)$ with a commutative ring
$\Pro^1$-spectrum $E$ and a Thom orientation
$th \in E^{2,1}(\Th(\mathcal{T}(1)))$.
The last assertion means
that there exists a unique monoidal morphism
$$
\varphi_{th}\colon \mathrm{MGL} \to E
$$
in the motivic stable homotopy category
$\mathrm{SH}(k)$
taking the class
$th^{\mathrm{MGL}}$
to
$th$ (see \cite{Vez} or \cite[Thm.~2.3.1]{UniThm} for precise formulations). 
That morphism
$\varphi_{th}$
gives rise to a functor transformation
$$
\bar \varphi_{th}\colon \mathrm{MGL}^{\ast}(X,U) \otimes_{\mathrm{MGL}^{0}(k)} E^{0}(k) \to
E^{\ast}(X,U).
$$
The pair
$(\mathrm{BGL}, th^{\mathrm{MGL}})$,
as it is described in
\cite{BGL} or in
\ref{OrientationsOfMGLandK},
induces the morphism $\bar \varphi$ mentioned in~(\ref{eq:2}).

There exist pairs $(E,th^E)$ such that
the induced monoidal morphism
$\bar \varphi_{th}$
is not an isomorphism.
For example one can take $E=\mathrm{H}\mathbb{Z}$ with
its canonical orientation, as mentioned in \cite[Introduction]{LM1}.
However for the pair
$(\mathrm{BGL}, th^{\mathrm{MGL}})$
it is an isomorphism.
A section
$s\colon K^{TT}_0 \to \mathrm{MGL}^{0,0}$
of the natural transformation
$\varphi\colon \mathrm{MGL}^{0,0} \to \mathrm{BGL}^{0,0}=K^{TT}_0$
is crucial for the proof of the main result.
This section is constructed in Section~\ref{AgeneralResult}.

\section{Recollection}
\label{Introduction}

Our main result relates
Voevodsky's algebraic cobordism theory
$\mathrm{MGL}^{\ast,\ast}$
to Quillen's
$K^{\prime}$-theory. We analyze these theories via their representing
objects in the motivic stable homotopy category $\SH(S)$.
Consider
\cite[Appendix]{BGL}
for the basic terminology, notation, constructions, definitions,
results on motivic homotopy theory. Nevertheless, here is a short
summary. 

\subsection{Motivic homotopy theory}

Let $S$ be a Noetherian separated finite-dimensional scheme $S$.
One may think of $S$ being the spectrum of a field or the integers.
A {\em motivic space
over\/} $S$ is a functor
\[ A\colon \SmOp/S \to \mathbf{sSet} \]
(see \cite[Appendix]{BGL}). The category of motivic spaces over $S$ is denoted
$\mathbf{M}(S)$. This definition of a motivic space is different from the one considered
by Morel and Voevodsky in \cite{MV} -- they consider only those simplicial presheaves
which are sheaves in the Nisnevich topology on $\Sm/S$. With our definition the
Thomason-Trobaugh $K$-theory functor obtained by using big vector bundles is a motivic space
on the nose. It is not a simplicial Nisnevich sheaf. This is the
reason why we prefer to work with the
above notion of ``space''.

We write
$\mathrm{H}^{\cm}_\bullet (S)$
for the pointed motivic homotopy category
and
$\mathrm{SH}^\cm(S)$
for
the stable motivic homotopy category over $S$ as constructed in
\cite[A.3.9, A.5.6]{BGL}.
By
\cite[A.3.11 resp.~A.5.6]{BGL}
there are canonical equivalences to
$\mathrm{H}_\bullet(S)$
of
\cite{MV}
resp.
$\mathrm{SH}(S)$
of
\cite{V1}.
Both
$\mathrm{H}^\cm_\bullet(S)$
and
$\mathrm{SH}^\cm(S)$
are equipped with
closed symmetric monoidal structures such that the
$\Pro^1$-suspension spectrum functor
is a strict symmetric monoidal functor
\[ \Sigma^\infty_{\Pro^1} \colon \mathrm{H}^\cm_\bullet(S)\to \mathrm{SH}^\cm(S). \]
Here
$\Pro^1$
is considered as a motivic space pointed by
$\infty \in \Pro^1$.
The symmetric monoidal structure $(\wedge,\mathbb{I}_S = \Sigma^\infty_{\Pro^1}S_+)$
on the homotopy category
$\mathrm{SH}^\cm(S)$ is constructed on the model category level by
employing the category $\MSS(S)$ of symmetric $\Pro^1$-spectra.
This symmetric monoidal category satisfies the properties required by
Theorem 5.6 of Voevodsky congress talk
\cite{V1}. From now on we will usually omit the superscript $(-)^\cm$.

Every $\Pro^1$-spectrum $E$ represents a cohomology theory on the category
of pointed motivic spaces. Namely, for a pointed motivic space $(A,a)$ set
\begin{eqnarray*}
  E^{p,q}(A,a)& =& \Hom_{\mathrm{SH}(S)}\bigl(\Sigma^\infty_{\Pro^1}(A,a), 
  \Sigma^{p,q}(E)\bigr)  \quad \mathrm{and} \\
  E^{\ast,\ast}(A,a) & = &\bigoplus_{p,q} E^{p,q}(A,a).
\end{eqnarray*}
This definition extends to motivic spaces via the functor $A\mapsto A_+$ which
adds a disjoint basepoint.
That is, for a non-pointed motivic space
$A$ set
$E^{p,q}(A)=E^{p,q}(A_+,+)$
and
$E^{\ast,\ast}(A)= \oplus_{p,q} E^{p,q}(A)$.

Every
$X \in \Sm/S$
defines a representable motivic space --  constant in the simplicial direction -- 
taking an smooth $S$-scheme $U$ to
$\mathrm{Hom}_{\Sm/S}(U,X)$. It is not possible in general to choose
a basepoint for representable motivic spaces. 
So we regard $S$-smooth varieties as motivic spaces (non-pointed)
and set
$$
E^{p,q}(X)=E^{p,q}(X_+,+).
$$

Given a $\Pro^1$-spectrum $E$ we will reduce the double grading on the cohomology
theory
$E^{\ast,\ast}$
to a grading by defining
$E^m = \oplus_{m = p-2q}E^{p,q}$
and
$E^{\ast}=\oplus_{m}E^m$.
{\it We often write\/}
$E^{\ast}(k)$
{\it for}
$E^{\ast}(\mathrm{Spec}(k))$
{\it below}.

A $\Pro^1$-ring spectrum is a monoid
$(E,\mu,e)$
in
$(\mathrm{SH(S)},\wedge, \mathbb{I}_S)$.
A commutative $\Pro^1$-ring
spectrum is a commutative monoid
$(E,\mu,e)$
in
$(\mathrm{SH(S)},\wedge, 1)$.

The cohomology theory $E^{\ast}$ defined by a
$\Pro^1$-ring spectrum is a ring cohomology theory.
The cohomology theory $E^{\ast}$ defined by a commutative
$\Pro^1$-ring spectrum is a ring cohomology theory,
however it is not necessarily graded commutative.
The cohomology theory $E^{\ast}$ defined by an oriented commutative
$\Pro^1$-ring spectrum (to be defined below)
is a graded commutative ring cohomology theory by \cite{PY}.

Occasionally a $\Pro^1$-ring spectrum $(E,\mu,e)$ might have a model
$(E',\mu',e')$ which is a symmetric $\Pro^1$-ring spectrum, that is, a
symmetric $\Pro^1$-spectrum $E'$ equipped with a strict multiplication $\mu'\colon E'\wedge E'
\to E'$
which is strictly associative and strictly unital for the unit
$e'\colon \Sigma^\infty_{\Pro^1}(S_+) \to E'$. This is the case for the algebraic
cobordism $\Pro^1$-ring spectrum $\mathrm{MGL}$, as described below. Such a model for the
algebraic $K$-theory $\Pro^1$-ring
spectrum $\mathrm{BGL}$ is currently not known to us.

For the rest of the paper let $k$ be a field and $S=\mathrm{Spec}(k)$.
Usually $S$ will be replaced by $k$ in the notation.
We work in this text with the algebraic cobordism
$\Pro^1$-spectrum
$\mathrm{MGL}$
and the algebraic $K$-theory
$\Pro^1$-spectrum
$\mathrm{BGL}$
as described in
\cite[Defn.~1.2.4]{BGL}
and
\cite[Sect.~2.1]{UniThm}
respectively.
The spectrum
$\mathrm{MGL}$
is a commutative ring
$\Pro^1$-spectrum
by that construction.
The spectrum
$\mathrm{BGL}$
is equipped with a structure of a commutative
$\Pro^1$-ring spectrum
as explained in
\cite[Thm.~2.1.1]{BGL}.

The $\Pro^1$-spectrum $\mathrm{BGL}$ has the following
underlying sequence of pointed motivic sequence:
$(\mathcal K_0, \mathcal K_1, \mathcal K_2, \dots )$.
Here are the main properties of $\mathrm{BGL}$ that will be used below:
\begin{itemize}
\item[(i)]
$\mathcal K_i = \mathcal K_0$ for all $i$ and
$\mathcal K_0$ is weakly equivalent to the Thomason-Trobaugh $K$-theory motivic space
$\mathbb{K}^{TT}$
\cite[Sect. 1.2]{BGL}.
\item[(ii)]
$\mathrm{BGL}$ is an $\Omega_{\Pro^1}$-spectrum; in particular, for any pointed
motivic space $A$
one has
$\mathrm{BGL}^{0,0}(A)=\Hom_{\mathrm{H}_\bullet(k)}(A,\mathcal{K}_0)$.
\item[(iii)]
There is a motivic weak equivalence
$w\colon \mathbb{Z} \times \mathrm{Gr}\to \mathcal K_0$
which classifies the tautological element
$\xi_\infty = \tau_{\infty}- \infty \in \mathbb{K}^{TT}_0(\mathbb{Z} \times \mathrm{Gr})$
\cite[Sect. 1.2]{BGL}.
\end{itemize}

Another property of $\mathrm{BGL}$ is that it represents
algebraic $K$-theory as a ring cohomology theory.
Let
$K^{TT}_*$
be Thomason-Trobaugh $K$-theory functor
\cite{TT}.
The morphism
$\Sigma^{\infty}\mathcal K_0 \to \mathrm{BGL}$
adjoint to the identity map
$id\colon \mathcal K_0 \to \mathcal K_0$
defines an isomorphism
$$
Ad\colon K^{TT}_{- *} \to \mathrm{BGL}^{*,0}
$$
of cohomology theories
on the category $\SmOp/k$
in the sense of
\cite{PSorcoh}.
This isomorphism respects the product structures by 
\cite[Cor. 2.2.4]{BGL}.

An invertible Bott element
$\beta \in \mathrm{BGL}^{2,1}(\Spec(k))$
is constructed in
\cite[Section 1.3]{BGL}.
For every pointed motivic space $A$ the morphism
\begin{eqnarray}
\label{BGLandTT}
\mathrm{BGL}^{*,0}(A) \otimes \mathrm{BGL}^{0}(\Spec(k)) \to
\mathrm{BGL}^{\ast,\ast}(A)
\end{eqnarray}
given by
$a \otimes \beta \mapsto a \cup \beta$
is a ring isomorphism by \cite[Sect.~1.3]{BGL}. Furthermore
$\mathrm{BGL}^{0}(\Spec(k))= \ZZ[\beta, \beta^{-1}]$
is the ring of Laurent polynomials on the Bott element
$\beta$. To say the same in a different way,
\begin{eqnarray}
\label{BGLlaurantAndTT}
\mathrm{BGL}^{*,0}(A)[\beta, \beta^{-1}]\cong
\mathrm{BGL}^{\ast,\ast}(A).
\end{eqnarray}
The special case
$A = X/(X\minus Z)$
where $X$ is a smooth $k$-variety and
$Z \subset X$ is a closed subset
implies the following result
\cite[Cor.~1.3.6]{BGL}.

\begin{corollary}
\label{LaurantPolinomial}
Let $X$ be a smooth $k$-scheme, $Z$ a closed subset of $X$
and $U = X\minus Z$ its open complement. Then
there are isomorphisms
\begin{eqnarray}
\label{BGLAndTTrings1}
K^{TT}_{- *, Z}(X)[\beta, \beta^{-1}] & \cong &\mathrm{BGL}^{\ast,\ast}(X/U) = \mathrm{BGL}^{\ast}(X/U)\\
\label{BGLAndTTrings2}
K^{TT}_{- *,Z}(X) & \cong &
 \mathrm{BGL}^{\ast,\ast}(X/U)/(\beta + 1)\mathrm{BGL}^{\ast,\ast}(X/U)
\end{eqnarray}
of ring cohomology theories on $\SmOp/k$ in the sense of
\cite{PSorcoh}.
\end{corollary}

We refer to
\cite{UniThm}
for a construction of the commutative
$\Pro^1$-ring spectrum
$\mathrm{MGL}$.
For the purposes of the work to be presented we will need to know
only two properties of $\mathrm{MGL}$, which we refer to as
Quillen universality and $\mathrm{BGL}$-cellularity
(see Section
\ref{AgeneralResult}
below).

\subsection{Oriented commutative ring spectra}
Following Adams and Morel we define an orientation of a commutative
$\Pro^1$-ring spectrum. However we prefer to use Thom classes instead of
Chern classes. Consider the pointed motivic space
$\Pro^{\infty}= \colim_{n\geq 0} \Pro^n$ having base point
$g_1 \colon S = \Pro^0 \hra \Pro^\infty$.

The tautological ``vector bundle''
$\mathcal{T}(1) = \mathcal{O}_{\Pro^\infty}(-1)$ is also known as the
Hopf bundle. It has zero section $z\colon \Pro^\infty \hra \mathcal{T}(1)$.
The fiber over the point
$g_1 \in \Pro^{\infty}$
is
$\mathbb{A}^1$.
For a vector bundle $V$ over a smooth
$k$-scheme $X$, with zero section $z\colon X \hra V$,
let the {\em Thom space\/} $\Th(V)$ of $V$ be
the Nisnevich sheaf associated to
the presheaf
$Y \mapsto V(Y)/\bigl(V\minus z(X)\bigr)(Y)$
on the Nisnevich site
$\Sm/k$.
In particular $\Th(V)$ is a pointed motivic space in the
sense of
\cite[Defn.~A.1.1]{BGL}.
It coincides with Voevodsky's Thom space \cite[p. 422]{V1},
since
$\Th(V)$
is already a Nisnevich sheaf.
The Thom space of the Hopf bundle is then defined as the colimit
$\Th(\mathcal{T}(1)) = \colim_{n\geq 0} \Th\bigl(\mathcal{O}_{\Pro^n}(-1)\bigr)$.
Abbreviate $T =\Th(\Aff^1_S)$.

Let $E$ be a commutative $\Pro^1$-ring spectrum. The unit gives rise
to an element
$1\in E^{0,0}(\Spec(k)_+)$.
Applying the
$\Pro^1$-suspension
isomorphism to that element we get an element
$\Sigma_{\Pro^1}(1) \in E^{2,1}(\Pro^1,\infty)$.
The canonical covering of $\Pro^1$ defines
motivic weak equivalences
\[ \xymatrix{\Pro^1 \ar[r]^-\sim & \Pro^1/\Aff^1 &
\Aff^1/\Aff^1\minus \lbrace 0 \rbrace = T \ar[l]_-\sim} \]
of pointed motivic spaces inducing isomorphisms
\[ 
E(\Pro^1,\infty) \leftarrow E(\Aff^1/\Aff^1\minus \lbrace 0 \rbrace) \rightarrow  
E(T).\]
Let
$\Sigma_{T}(1)$
be the image of
$\Sigma_{\Pro^1}(1)$
in
$E^{2,1}(T)$.

\begin{definition}
\label{OrientationViaThom}
Let $E$ be a commutative ring $\Pro^1$-spectrum.
A {\em Thom orientation of\/}
$E$ is an element
$th \in E^{2,1}(\Th(\mathcal T(1))$
such that its restriction to the Thom space of the fibre over the distinguished
point coincides with the element
$\Sigma_{T}(1) \in E^{2,1}(T)$.
A {\em Chern orientation of\/}
$E$ is an element
$c \in E^{2,1}(\Pro^{\infty})$
such that
$c|_{\Pro^1}= - \Sigma_{\Pro^1}(1)$.
{\em An orientation} of $E$ is either a Thom orientation or a Chern orientation.
One says that a Thom orientation $th$ of $E$ coincides
with a Chern orientation $c$ of $E$ provided that
$c = z^*(th)$ or equivalently
the element
$th$ coincides with the one
$th(\mathcal O(-1))$
given by
(\ref{ThomClass})
below.
\end{definition}


\begin{remark}
\label{ThomAndChern}
The element $th$ should be regarded as the Thom class of the tautological
line bundle
$\mathcal{T}(1)= \mathcal O(-1)$
over
$\Pro^{\infty}$.
The element $c$ should be regarded as the Chern class of the tautological
line bundle
$\mathcal{T}(1)= \mathcal O(-1)$
over
$\Pro^{\infty}$.
\end{remark}


\begin{example}
\label{OrientationsOfMGLandK}
The following orientations given right below are relevant for our work.
Here $\mathrm{MGL}$ denotes the $\Pro^1$-ring spectrum
representing algebraic cobordism obtained in
\cite[Defn 2.1.1]{UniThm}
and $\mathrm{BGL}$ denotes the $\Pro^1$-ring spectrum
representing algebraic $K$-theory
constructed in
\cite[Theorem 2.2.1]{BGL}.
\begin{itemize}
\item
Let
$u_1: \Sigma^{\infty}_{\Pro^1}(\Th(\mathcal{T}(1)))(-1) \to \mathrm{MGL}$
be the canonical map of
$\Pro^1$-spectra. Set
$th^{\mathrm{MGL}} =u_1 \in \mathrm{MGL}^{2,1}(\Th(\mathcal{T}(1)))$.
Since
$th^{\mathrm{MGL}}|_{\Th(\mathbf{1})}= \Sigma_{\Pro^1}(1)$
in
$\mathrm{MGL}^{2,1}(\Th(\mathbf{1}))$,
the class
$th^{\mathrm{MGL}}$
is an orientation of
$\mathrm{MGL}$.
\item
Set
$c^{\mathrm{BGL}} =(- \beta) \cup ([\mathcal O]-[\mathcal O(1)])
\in \mathrm{BGL}^{2,1}(\Pro^{\infty})$.
The relation
(11)
from
\cite{BGL}
shows that
the class $c$ is an orientation of
$\mathrm{BGL}$.
Let
$th^{\mathrm{BGL}}$
be the equivalent Thom orientation.
\end{itemize}
\end{example}


\section{Oriented cohomology theories}
\label{OrientedSpectraAndTheory}

Let $(E,c)$ be an oriented commutative $\Pro^1$-ring spectrum
(See \ref{OrientationViaThom}). In this section we compute
the $E$-cohomology of infinite
Grassmannians. The results
are the expected ones -- see Theorem
\ref{CohomologyOfGr}.

The oriented $\Pro^1$-ring spectrum $(E,c)$ defines an oriented cohomology theory on
$\SmOp/k$ in the sense of
\cite[Defn.~3.1]{PSorcoh}
as follows.
The restriction of the functor $E^{\ast,\ast}$ to the category
$\SmOp/k$ is a ring cohomology theory.
By
\cite[Th.~3.35]{PSorcoh}
it remains to construct
a Chern structure on
$E^{\ast,\ast}|_{\SmOp/k}$
in the sense of
\cite[Defn.3.2]{PSorcoh}.
Let
$\mathrm{H}_\bullet(k)$
be the homotopy category of pointed motivic spaces over $k$.
The functor isomorphism
$\Hom_{\mathrm{H_\bullet}(k)}(- , \Pro^{\infty}) \to \mathrm{Pic}(-)$
on the category
$\Sm/k$
provided by
\cite[Thm.~4.3.8]{MV}
sends the class of the identity map
$\Pro^\infty \to {\Pro^{\infty}}$ to the class of the tautological line bundle
$\mathcal{O}(-1)$
over
$\Pro^{\infty}$.
For a line bundle $L$ over
$X\in \Sm/k$ let
$[L]$ be the class of $L$ in the group
$\mathrm{Pic}(X)$.
Let
$f_L\colon X \to \Pro^{\infty}$
be the morphism in
${\mathrm{H}_\bullet(k)}$ corresponding to
the class $[L]$
under the functor isomorphism above.
For a line bundle $L$ over $X\in \Sm/k$
set
$c(L)=f^\ast_L(c) \in E^{2,1}(X)$.
Clearly,
$c(\mathcal{O}(-1))=c$.
The assignment
$L/X \mapsto c(L)$
is a Chern structure on
$E^{\ast,\ast}|_{\SmOp/k}$
since
$c|_{\Pro^1}= - \Sigma_{\Pro^1}(1) \in E^{2,1}(\Pro^1,\infty)$.
With that Chern structure
$E^{\ast,\ast}|_{\SmOp/k}$
is an oriented ring cohomology theory
in the sense of
\cite{PSorcoh}.
In particular,
$(\mathrm{BGL},c^{\mathrm{BGL}})$
defines an oriented ring cohomology theory on
$\SmOp/k$.

Given this Chern structure, one
obtains a theory of Thom classes
$V/X \mapsto th(V) \in E^{2\mathrm{rank}(V),\mathrm{rank}(V)}\bigl(\Th_X(V)\bigr)$
on the cohomology theory
$E^{\ast,\ast}|_{\SmOp/k}$
in the sense of
\cite[Defn.~3.32]{PSorcoh} as follows.
There is a unique theory of Chern classes
$V \mapsto c_i(V) \in E^{2i,i}(X)$
such that for every line bundle $L$ on $X$ one has
$c_1(L)=c(L)$. For a rank $r$ vector bundle
$V$ over $X$ consider the vector bundle
$W:= {\mathbf{1}} \oplus V$
and the associated projective vector bundle
$\Pro(W)$
of lines in $W$.
Set
\begin{eqnarray}
\label{ThomBarClass}
\bar th(V)= c_r(p^\ast(V) \otimes \mathcal{O}_{\Pro(W)}(1)) \in E^{2r,r}(\Pro(W)).
\end{eqnarray}
It follows from
\cite[Cor.~3.18]{PSorcoh}
that the support extension map
\[E^{2r,r}\bigl(\Pro(W)/(\Pro(W)\smallsetminus \Pro(\mathbf{1}))\bigr)
\to E^{2r,r}\bigl(\Pro(W)\bigr)\]
is injective and
$\bar th(E) \in E^{2r,r}\bigl(\Pro(W)/(\Pro(W)\smallsetminus \Pro(\mathbf{1}))\bigr) $.
Set
\begin{eqnarray}
\label{ThomClass}
th(E)= j^\ast(\bar th(E)) \in E^{2r,r}\bigl(\Th_X(V)\bigr),
\end{eqnarray}
where
$j\colon \Th_X(V) \to \Pro(W)/ (\Pro(W) \smallsetminus \Pro({\bf 1}))$
is the canonical motivic weak equivalence of pointed motivic spaces
induced by the open embedding $V\hra \Pro(W)$.
The assignment $V/X \mapsto th(V)$ is a theory of Thom classes
on the cohomology theory $E^{\ast,\ast}|_{\SmOp/k}$
by the proof of
\cite[Thm.~3.35]{PSorcoh}. So the Thom classes are natural,
multiplicative and satisfy the following Thom isomorphism property.

\begin{theorem}
\label{ThomIsomorphism}
For a rank $r$ vector bundle
$p\colon V \to X$
on $X\in \Sm/k$ with zero section $z\colon X\hra V$, the map
\[- \cup th(V)\colon  E^{\ast,\ast}(X) \to E^{\ast+2r,\ast+r}\bigl(V/(V\minus z(X))\bigr) \]
is an isomorphism of two-sided
$E^{\ast,\ast}(X)$-modules,
where
$-\cup th(V)$
is written for the composition map
$\bigl(-\cup th(V)\bigr) \circ p^\ast$.
\end{theorem}

\begin{proof}
  See \cite[Defn.~3.32.(4)]{PSorcoh}.
\end{proof}


Let $\Gr(n,n+m)$ be the Grassmann scheme of $n$-dimensional linear
subspaces of $\Aff^{n+m}_S$. The closed embedding
$\Aff^{n+m} = \Aff^{n+m}\times \{0\} \hra \Aff^{n+m+1}$
defines a closed embedding
\begin{equation}
\label{eq:1}
\Gr(n,n+m)\hra \Gr(n,n+m+1).
\end{equation}
The tautological vector bundle
is denoted
$\mathcal{T}(n,n+m)\to \Gr(n,n+m)$.
The closed embedding~(\ref{eq:1})
is covered by a map of vector bundles
$\mathcal{T}(n,n+m)\hra \mathcal{T}(n,n+m+1)$.
Let
$\Gr(n) = \colim_{m\geq 0} \Gr(n,n+m)$, $\mathcal{T}(n) = \colim_{m\geq 0} \mathcal{T}(n,n+m)$
and
$\Th(\mathcal{T}(n)) = \colim_{m\geq 0} \Th(\mathcal{T}(n,n+m))$.
These colimits are taken in the category of motivic spaces over $S$.

\begin{remark}
\label{FiniteGrassmannians}
It is not difficult to prove that 
$E^{\ast,\ast}(\Gr(n,n+m))$
is multiplicatively generated by the Chern classes 
$c_i(\mathcal{T}(n,n+m))$ 
of the tautological vector bundle
$\mathcal{T}(n,n+m)$. This proves the surjectivity of the pull-back maps
$E^{\ast,\ast}(\Gr(n,n+m+1)) \to E^{\ast,\ast}(\Gr(n,n+m))$ 
and shows that the canonical map
$E^{\ast,\ast}(\Gr(n)) \to {\varprojlim}E^{\ast,\ast}(\Gr(n,n+m))$
is an isomorphism. Thus for each $i$ there exists a unique element 
$c_i= c_i(\mathcal T(n)) \in E^{2i,i}(\Gr(n))$
which for each $m$ restricts to the element 
$c_i(\mathcal{T}(n,n+m))$ under the obvious pull-back map. 
\end{remark}

\begin{theorem}
\label{CohomologyOfGr}
Let
$c_i= c_i(\mathcal T(n)) \in E^{2i,i}(\Gr(n))$
be the element defined just above. 
Then
$$
E^{\ast,\ast}(\Gr(n))= E^{\ast,\ast}(k)[[c_1,c_2, \dots, c_n]]
$$
is the formal power series on the $c_i$'s.
The inclusion
$i\colon \Gr(n) \hra \Gr(n+1)$
satisfies
$i^\ast(c_m)=c_m$
for $m < n+1$ and
$i^\ast(c_{n+1})=0$.
\end{theorem}

\begin{proof}
See
\cite[Thm. 2.0.6]{UniThm}.
\end{proof}

\begin{remark}
\label{FirstChernOfK}
For a smooth variety $X$ and a vector bundle $E$ over $X$ the class $c_1(E)$
is additive with respect to short exact sequences of vector bundles. Thus
it defines a homomorphism
$K^{TT}_0(X) \to E^{2,1}(X)$.
Moreover that homomorphism is natural in $X$.
For an element
$\alpha \in K^{TT}_0(X)$
we will write
$c_1(\alpha)$
for the image of
$\alpha$.
Now take the space
$\Gr(n)$ and recall that the map
$
\mathbb{K}^{TT}_0(\Gr(n)) \to {\varprojlim}K^{TT}_0(\Gr(n,n+m))
$
is an isomorphism by 
\cite[Sect. 1.2]{BGL}
and the map
$E^{\ast}(\Gr(n)) \to {\varprojlim}E^{\ast}(\Gr(n,n+m))$
is an isomorphism by Remark
\ref{FiniteGrassmannians}. 
Thus we have a homomorphism
$\mathbb{K}^{TT}_0(\Gr(n)) \to E^{2,1}(\Gr(n))$.
In the same way we may get a homomorphism
$$
\mathbb{K}^{TT}_0(\mathbb Z \times \Gr) \to E^{2,1}(\mathbb Z \times \Gr),
$$
where
$\Gr = \colim_{n\geq 0} \Gr(n)$. For an element
$\alpha \in \mathbb{K}^{TT}_0(\mathbb Z \times \Gr)$
we will write
$c_1(\alpha)$
for its image in
$E^{2,1}(\mathbb Z \times \Gr)$.
\end{remark}

\subsection{Universality and cellularity}
\label{AgeneralResult}
The main result of this Section is Theorem
\ref{AlmostlyMain}.

\begin{definition}[Universality Property]
\label{UniversalityDefn}
Let
$(\mathrm{U},\mathrm{u})$
be an oriented
commutative ring
$\Pro^1$-spectrum over $S$.
We say that
$(\mathrm{U},\mathrm{u})$
is  {\em Quillen universal\/}  if
for every commutative ring
$\Pro^1$-spectrum $E$ over $S$
the assignment
$\varphi \mapsto \varphi(\mathrm{u}) \in \mathrm{U}^{2,1}(\Th(\mathcal T(1)))$
identifies the set of monoid homomorphisms
\begin{eqnarray}
\label{OrientationMap}
\varphi\colon \mathrm{U} \to \mathrm{E}
\end{eqnarray}
in the motivic stable homotopy category
$\mathrm{SH}(S)$
with the set of orientations of $\mathrm{E}$.
\end{definition}

\begin{remark}
The Universality Theorem
(\cite{Vez}
or
\cite{UniThm}) implies that
the
$\Pro^1$-spectrum
$\mathrm{MGL}$, equipped with its canonical
orientation
$th^{\mathrm{MGL}}$
from
\ref{OrientationsOfMGLandK},
is Quillen universal.
\end{remark}

If $E$ is a commutative $\Pro^1$-ring spectrum over $k$ and
$A$ is a pointed motivic space over $k$, $E^{\ast}(A)$ is
an $E^0(k)$-module in a natural way. A monoid homomorphism $\phi\colon E_1 \to E_2$
induces an $E^0(k)$-module homomorphism $E_1^\ast(A) \to E_2^\ast(A)$.
In particular, if
$(\mathrm{U},\mathrm{u})$
is a Quillen universal oriented commutative ring
$\Pro^1$-spectrum over $k$
and
$(\mathrm{E},\mathrm{th})$ is an
oriented commutative ring
$\Pro^1$-spectrum over $k$, the
monoid homomorphism
\begin{eqnarray}
\label{MainMOrphism}
\varphi\colon \mathrm{U} \to \mathrm{E}
\end{eqnarray}
in
$\mathrm{SH}(k)$
induces the homomorphism
\begin{eqnarray}
\label{CFhomPeriodic}
\bar \varphi_{A}\colon \mathrm{U}^{\ast}(A) \otimes_{\mathrm{U}^{0}(k)}
\mathrm{E}^{0}(k) \to
\mathrm{E}^{\ast}(A)
\end{eqnarray}
and in particular the homomorphism
\begin{eqnarray}
\label{CFhomPeriodiczero}
\bar \varphi^0_{A}\colon \mathrm{U}^{0}(A) \otimes_{\mathrm{U}^{0}(k)}
\mathrm{E}^{0}(k) \to
\mathrm{E}^{0}(A).
\end{eqnarray}
Both are natural in $A$.

From now on we will insert
$(\mathrm{BGL}, \mathrm{th}^K)$
for
$(\mathrm{E},\mathrm{th})$
(see Example
\ref{OrientationsOfMGLandK}).
Set
\[  \bar {\mathrm{U}}^{\ast}(A) = 
\mathrm{U}^{\ast}(A) \otimes_{\mathrm{U}^{0}(k)} \mathrm{BGL}^{0}(k) \quad
\mathrm{and}\quad
\bar {\mathrm{U}}^{0}(A)  
\mathrm{U}^{0}(A) \otimes_{\mathrm{U}^{0}(k)} \mathrm{BGL}^{0}(k).\]

\begin{definition}[Weakly BGL-cellular]
\label{MGLWeaklyCellular}
Let $(\mathrm{U},u)$ be a Quillen universal
$\Pro^1$-ring spectrum, and let
$\bar\phi^0_A\colon \mathrm{U}^0(A) \otimes_{\mathrm{U}^0(k)} \mathrm{BGL}^0(k) \to \mathrm{BGL}^0(A)$
be the homomorphism induced by the orientation
$th^{\mathrm{BGL}}$ on $\mathrm{BGL}$
(see Example
\ref{OrientationsOfMGLandK}).
Then $(\mathrm{U},u)$ is
is called {\em weakly\/} $\mathrm{BGL}$-{\em cellular\/}
if there exists an integer $N$
such that
the map
$\bar \varphi^0_{\mathrm{U}_n}$
is an isomorphism for
$n \geq N$,
where
$\mathrm{U}_n$ is the $n$-th term of the $\Pro^1$-spectrum
$\mathrm{U}$.
\end{definition}

A pointed motivic space $A$ is called {\em small\/} if
the covariant functor
\[\Hom_{\SH(S)}(\Sigma^\infty_{\Pro^1} A,-) \] on
$\mathrm{SH}(S)$ commutes with arbitrary coproducts.

\begin{theorem}
\label{AlmostlyMain}
Let
$(\mathrm{U},\mathrm{u})$
be a Quillen universal oriented commutative
$\Pro^1$-ring spectrum over a field $k$.
Suppose
$(\mathrm{U},\mathrm{u})$
is weakly $\mathrm{BGL}$-cellular.
Then
the homomorphism
$\bar \varphi_A$
is an isomorphism for all small pointed motivic spaces $A$.
\end{theorem}


\begin{proof}
The proof consists of several steps.
Our first aim is to prove that the homomorphisms
$\bar \varphi^0_{A}$
are isomorphisms, where $A$ is a small pointed motivic space.
First we construct a section of the natural
transformation
\[
\varphi^{0,0}\colon \mathrm{U}^{0,0} \to \mathrm{BGL}^{0,0}
\]
of functors on the category of 
small pointed motivic spaces.
Recall
that for every oriented commutative
$\Pro^1$-ring spectrum
$(E,th)$
the ring cohomology theory
$E^{\ast,\ast}|_{\SmOp/k}$
is an oriented cohomology theory on the category
$\SmOp/k$
(see Section
\ref{OrientedSpectraAndTheory}).
Let
$\mathbb{F}_{E,th}$
be the induced commutative formal group law
over the ring
$E^0(k)$.
Let
$\Omega$
be the complex cobordism ring
and let
$l_{E,th}\colon \Omega \to E^0(k)$
be the unique ring homomorphism
sending the universal formal group
$\mathbb{F}_{\Omega}$
to
$\mathbb{F}_{E,th}$.

Since
$c_1 \in E^{2,1}$
the coefficients $a_{ij}$ of the formal group law
$\mathbb{F}_{E,th}$
are in
$E^{-2(i+j-1),-(i+j-1)}(k)$. Thus the homomorphism
$l_{E,th}$
is grade preserving, that is, it takes
$\Omega^{2i}$ into $E^{2i,i}(k)$ for any $i$.

Set
\begin{eqnarray}
\label{ClassOfPro}
[\Pro^n]_{E}= l_{E,th}([\mathbb C \mathbb P^n]) \in E^{-2n,-n}(k),
\end{eqnarray}
where
$[\mathbb C \mathbb P^n]$
is the class of the complex projective space
$\mathbb C \mathbb P^n$
in
$\Omega$.
Although the class
$[\Pro^n]_{E}$
depends on the orientation class $th$, we
use the notation
$[\Pro^n]_{E}$
instead. If
$(E^{\prime},th^{\prime})$
is another oriented commutative
$\Pro^1$-ring spectrum and
$\psi\colon E \to E^{\prime}$
is a monoid homomorphism in the
category
$\mathrm{SH}(S)$
which preserves orientation classes, then it
sends the formal group law
$\mathbb{F}_{E,th}$
to
$\mathbb{F}_{E^{\prime},th^{\prime}}$.
In particular
$
\psi([\Pro^n]_{E})=[\Pro^n]_{E^{\prime}}.
$
Applying this observation to the monoid homomorphism
$\varphi$
one obtains
\[
\varphi([\Pro^1]_{\mathrm{U}})= [\Pro^1]_{\mathrm{BGL}}.
\]
To compute
$[\Pro^1]_{\mathrm{BGL}}$
recall that
the coefficient at
$XY$ in the formal group law
$\mathbb F_{\Omega}$
coincides with the class\
$-[\mathbb C\mathbb P^1]$
in
$\Omega$.
The formal group law
$\mathbb F_{\mathrm{BGL}}$
coincides with
$X+Y+\beta^{-1}XY$,
since
$c^{\mathrm{BGL}}(L)=([\bf 1]- [L^{\vee}])(- \beta)$.
Thus the equality
\[ [\Pro^1]_{\mathrm{BGL}}= - \beta^{-1} \]
holds.
{\em We are ready to construct a section}.

Let
$\Gr(n)$
be the pointed motivic space described right above Theorem
\ref{CohomologyOfGr}.
Set
$\Gr = \colim_{n\geq 0} \Gr(n)$.
Consider the unique element
$\infty - \tau_{\infty}^{\vee} \in \mathbb{K}^{TT}_0(\mathbb{Z} \times \mathrm{Gr})$
such that for each integer $m$ its restriction to the subspace
$\lbrace m\rbrace  \times \mathrm{Gr}(n,2n) \subseteq \mathbb{Z} \times \mathrm{Gr}$
coincides with the element
$-m+n-[\tau^{\vee}_n] \in K^{TT}_0(\mathrm{Gr}(n,2n))$
(compare with description of the element
$\xi_\infty=\tau_{\infty}-\infty$
in
\cite[Sect. 1.2]{BGL}).
Let
$c^{\mathrm{U}}_1(\infty - \tau_{\infty}^{\vee})$
be the image of
$\infty - \tau_{\infty}^{\vee}$
in
$E^{2,1}(\mathbb Z \times \mathrm{Gr})$
as described in Remark
\ref{FirstChernOfK}.
Consider the map
\begin{eqnarray}
\label{Section}
s\colon \Sigma_{\Pro^1}^{\infty}(\mathbb Z \times \mathrm{Gr}) \to \mathrm{U}
\end{eqnarray}
in the stable homotopy category category
$\mathrm{SH}(S)$
given by the element
\[
c^{\mathrm{U}}_1(\infty - \tau_{\infty}^{\vee}) \cup [\Pro^1]_{\mathrm{U}}
\in \mathrm{U}^{0,0}(\mathbb Z \times \mathrm{Gr}).
\]

\begin{claim}
\label{PhiOfTheClass}
One has
$\varphi(c^{\mathrm{U}}_1(\infty - \tau_{\infty}^{\vee}) \cup [\Pro^1]_{\mathrm{U}})=
\tau_{\infty}- \infty
\in \mathrm{BGL}^{0,0}(\mathbb Z \times \mathrm{Gr})$.
\end{claim}
In fact,
\begin{eqnarray*}
\varphi(c^{\mathrm{U}}_1(\infty - \tau_{\infty}^{\vee}) \cup [\Pro^1]_{\mathrm{U}}) & = &
c^{\mathrm{BGL}}_1(\infty - \tau_{\infty}^{\vee}) \cup [\Pro^1]_{\mathrm{BGL}} \\ &  = &
(\infty - \tau_{\infty}) \cup \beta \cup (- \beta^{-1})  \\
       & = & \tau_{\infty}- \infty
\end{eqnarray*}
Claim \ref{PhiOfTheClass} shows that the composition
\[ \varphi \circ s\colon \Sigma_{\Pro^1}^{\infty}(\mathbb Z \times \mathrm{Gr}) \to \mathrm{BGL} \]
coincides with the adjoint of the
motivic
weak equivalence
$w\colon \mathbb Z \times \mathrm{Gr} \to \mathcal K_0$
from Section
\ref{Introduction}.

This fact together with properties
$(ii)$ and $(iii)$ of
$\mathrm{BGL}$ mentioned in
Section
\ref{Introduction}
shows that for every
pointed motivic space $A$ the map
\[
s_{A}\colon \mathrm{BGL}^{0,0}(A)= [A, \K_0]=
[A, \mathbb Z \times \mathrm{Gr}] \to [\Sigma_{\Pro^1}^{\infty}(A), \mathrm{U}]=
\mathrm{U}^{0,0}(A)
\]
is a {\it section} of the map
$\varphi^{0,0}_{A}\colon \mathrm{U}^{0,0}(A) \to \mathrm{BGL}^{0,0}(A)$.
Moreover, the section
$s_{A}$
is natural in $A$.
Let 
\[\overline{\mathrm{U}}^\ast(A) = 
\mathrm{U}(A)\tensor_{\mathrm{U}^0(k)}\mathrm{BGL}^0(k)\] 
and
recall that $\bar \phi_A\colon  \overline{\mathrm{U}}^\ast(A) \to \mathrm{BGL}^\ast(A)$
is the induced $\mathrm{BGL}^0(k)$-module homomorphism.
To extend the section $s$ to a section
$\bar s^0\colon \mathrm{BGL}^0 \to \overline{\mathrm{U}}^0$
of the natural transformation
$\bar \varphi^0\colon \overline{\mathrm{U}}^0 \to \mathrm{BGL}^0$
of functors on pointed motivic spaces.
Note that
\[
\mathrm{BGL}^0=\mathrm{BGL}^{0,0}[\beta,\beta^{-1}]
\]
for the Bott element
$\beta \in \mathrm{BGL}^{2,1}(k)$ (see
(\ref{BGLlaurantAndTT})).
Thus
for every pointed motivic space $A$,
every homogeneous element
$\alpha \in \mathrm{BGL}^0(A)$
can be presented in a unique way in the form
$a \cup \beta^i$
with
$a \in \mathrm{BGL}^{0,0}(A)$.
Define
\begin{eqnarray}
\label{Salgebraic}
\bar s^0_{A}\colon \mathrm{BGL}^0 \to \overline{\mathrm{U}}^0
\end{eqnarray}
by
$\bar s_{A}^0(a \cup \beta^i)= s_{A}(a) \otimes \beta^i \in \overline{\mathrm{U}}^0(A)$,
where
$a \in \mathrm{BGL}^{0,0}(A)$.
It is immediate that $s_A^0$ is natural in $A$.
The computation
\[
\bar \varphi^0_A(\bar s^0(a \cup \beta^i))=
\bar \varphi^0_A(s(a) \otimes \beta^i)=
\varphi(s(a)) \cup \beta^i=
a \cup \beta^i \]
proves the following
\begin{claim}
The map $\bar s^0_{A}$
is a section of
$\bar \varphi^0_{A}$.
\end{claim}


If for a pointed motivic space $A$ the map
$\bar \varphi^0_{A}$
is an isomorphism,
then
$\bar s^0_{A}$
is an isomorphism inverse to
$\bar \varphi^0_{A}$.
In particular, one has
$\bar s^0_{A} \circ \bar \varphi^0_{A}= \id$.

The homomorphism
$\bar \varphi^0_{A}$
is an isomorphism for the pointed motivic spaces
$\mathrm{U}_n$
with
$n \geq N$,
since
$\mathrm{U}$ is weakly
$\mathrm{BGL}$-cellular.
The class
$[u_n] \in \mathrm{U}^{2n,n}(\mathrm{U}_n, \ast)$
of the canonical morphism
$u_n\colon \Sigma^{\infty}_{\Pro^1}\mathrm U_n(-n) \to \mathrm{U}$
then satisfies the following relation:

\begin{eqnarray}
\label{KeyRelationForU}
(\bar s^0_{\mathrm{U}_n} \circ \varphi^0_{\mathrm{U}_n})([u_n])
= [u_n] \otimes 1 \in
\overline{\mathrm{U}}^0(\mathrm{U}_n).
\end{eqnarray}

Now we are ready to check that
$\bar \varphi^0_A$
is an isomorphism for all small pointed motivic spaces.
Recall that for a
small pointed motivic space $A$
there is a canonical isomorphism of the form
\begin{eqnarray}
\label{IndLimit}
\mathrm{U}^{2i,i}(A)= \colim_n [\Sigma^{2n,n}(A),
\mathrm{U}_{i+n}]_{\mathrm{H}_\bullet(S)}
\end{eqnarray}
where
$\Sigma^{2n,n}= \Sigma^n_{\Pro^1}$.
This isomorphism implies that for every element
$a \in \mathrm{U}^{2i,i}(A)$
there exists an integer
$n \geq 0$ such that
$\Sigma^{2n,n}(a)= f^\ast([u_n])$
for an appropriate map
$f\colon \Sigma^{2n,n}(A) \to \mathrm{U}_{i+n}$
in the homotopy category
$\mathrm{H}_\bullet(S)$. Here
$\Sigma^{2n,n}(a)$
is the $n$-fold
$\Sigma_{\Pro^1}$-suspension
of $a$.

The surjectivity of
$\bar \varphi^0_{A}$
is clear, since
$\bar s^0_{A}$
is its section.
It remains to check the injectivity of
$\bar \varphi^0_{A}$.
Take a homogeneous element
$\alpha \in \overline{\mathrm{U}}^{2i,i}(A) \subseteq \overline{\mathrm{U}}^0(A)$
such that
$\bar \varphi^0_{A}(\alpha)=0$.
It has the form
$\alpha= a \otimes \beta^m$
for a homogeneous element
$a \in \mathrm{U}^0(A)$.
Since the element
$\beta$
is invertible in
$\mathrm{BGL}^{\ast,\ast}(k)$, one concludes
$\varphi^0_{A}(a)=0$.

Choose an integer $n \geq 0$ such that
$\Sigma^{2n,n}(a)= f^\ast([u_n])$.
The map $\varphi$ of $\Pro^1$-spectra
respects the suspension isomorphisms.
Thus
$\varphi_{\Sigma^{2n,n}A}(\Sigma^{2n,n}(a))= \Sigma^{2n,n}(\varphi_A(a))=0$
and
$(\bar s^0_{\Sigma^{2n,n}A} \circ \varphi_{\Sigma^{2n,n}A})(\Sigma^{2n,n}(a))=0$
too.
The chain of relations in
$\overline{\mathrm{U}}^0(\Sigma^{2n,n}A)$
given by
\begin{eqnarray*}
  0 & = & \bigl(\bar s_{\Sigma^{2n,n}A}^0 \circ \varphi_{\Sigma^{2n,n}A}\bigr)
  \bigl(\Sigma^{2n,n}(a)\bigr)= \bigl(\bar s_{\Sigma^{2n,n}A}^0 \circ
  \varphi_{\Sigma^{2n,n}A}\bigr)\bigl(f^\ast([u_n])\bigr) \\
   & = & f^\ast\bigl((\bar s^0_{\mathrm{U}_{n+i}} \circ \varphi_{\mathrm{U}_{n+i}})
   ( [u_n] ) \bigr) =f^\ast([u_n] \otimes 1)= f^\ast([u_n]) \otimes 1 \\
& = &\Sigma^{2n,n}(a) \otimes 1
\end{eqnarray*}
implies that
$\Sigma^{2n,n}(a \otimes 1)= \Sigma^{2n,n}(a) \otimes 1=0$.
Because the $n$-fold suspension map
\[
\Sigma^{2n,n}\colon \overline{\mathrm{U}}^0(A) \to
\overline{\mathrm{U}}^0(\Sigma^{2n,n} A) \]
is an isomorphism,
$a \otimes 1=0$
in
$\overline{\mathrm{U}}^0(A)=\overline{\mathrm{U}}^0(A)$.
This proves the injectivity and hence the bijectivity of
$\bar \varphi^0_{A}$
for all small motivic spaces.

To prove that
$\bar \varphi_{A}$
is an isomorphism for
all small motivic spaces
we will use
the fact that
$\bar \varphi_{A}$
respects the
$\Pro^1$-suspension isomorphisms.
For every integer
$i \in \mathbb{Z}$
choose an integer
$n \geq 0$
with
$n \geq i$.
Then for a pointed motivic space $A$ one may form the suspension
$\mathbb{G}^{\wedge n}_m \wedge S^{n-i}_s \wedge A = S^{n,n}\wedge S^{n-i,0}\wedge A$
in the category of pointed motivic spaces, which supplies the commutative diagram
\[
\xymatrix{
\mathrm{BGL}^{i}(A) \ar[r]^-{\Sigma^{2n,n}}_-\cong &
\mathrm{BGL}^{i}(S^{2n,n} \wedge A) &
\mathrm{BGL}^0(S^{n,n}\wedge S^{n-i,0}\wedge A) \ar[l]_-{\Sigma^{i,0}}^-\cong \\
\mathrm{U}^{i}(A) \ar[r]^-{\Sigma^{2n,n}}_-\cong \ar[u]^-{\bar \varphi^i_A}&
\mathrm{U}^{i}(S^{2n,n} \wedge A) \ar[u]^-{\bar \varphi^i_{S^{2n,n}\wedge A}}&
\mathrm{U}^0(S^{n,n}\wedge S^{n-i,0}\wedge A) \ar[l]_-{\Sigma^{i,0}}^-\cong
\ar[u]_-{\bar \varphi^0_{S^{n,n}\wedge S^{n-i,0}\wedge A}}^-\cong} \]
with the suspension isomorphisms
$\Sigma^{2n,n} = \Sigma^n_{\Pro^1}$
and
$\Sigma^{i,0}$.
The map
$\bar \varphi^0_B$
is an isomorphism for $B$ a small pointed motivic space, hence so is
$\bar \varphi^{i}_A$.
We proved that the map
$\bar \varphi_{A}$
is an isomorphism for $A$ being small.
\end{proof}

\subsection{The $\mathrm{BGL}$-cellularity of algebraic cobordism}

\begin{theorem}
\label{Preliminary}
The oriented commutative ring
$\Pro^1$-spectrum
$(\mathrm{MGL}, th^\mathrm{MGL})$
from Example
\ref{OrientationsOfMGLandK}
is weakly
$\mathrm{BGL}$-cellular.
\end{theorem}

\begin{proof}
We will prove that the homomorphism
$\bar \varphi^0_{A}$
is an isomorphism for ${A}$ being
one of the pointed motivic spaces
$\mathrm{Spec}(k)_+$, $\Pro^{\infty}_+$, $\Gr(n)_+$
and
$\Th(\mathcal{T}_n) = \mathrm{MGL}_n$.

The map
$\bar \varphi^0_{k}$
is an isomorphism,
since it is the identity map.
The case $n=1$ of Theorem
~\ref{CohomologyOfGr}
implies that
$\overline{\mathrm{MGL}}^{\ast}(\Pro^{\infty})=
\overline{\mathrm{MGL}}^{\ast}(k)[[c^{\mathrm{MGL}}]]$,
whence
\[ \overline{\mathrm{MGL}}^0(\Pro^{\infty})=
\overline{\mathrm{MGL}}^0(k)[[c^{\mathrm{MGL}}]]\]
(the formal power series on the first Chern class
$c^{\mathrm{MGL}}$
of the tautological line bundle
$\mathcal O(-1)$).
The same holds for
$\mathrm{BGL}$. Namely
\[ \mathrm{BGL}^0(\Pro^{\infty})= \mathrm{BGL}^0(k)[[c^{\mathrm{BGL}}]]. \]
By its definition the morphism
$\varphi$
takes the orientation class
$th^{\mathrm{MGL}}$
to the orientation class
$th^K$
and so it preserves the first Chern class.
Whence the map
$\bar \varphi^0_{\Pro^{\infty}}$
coincides with the map of formal power series
induced by the isomorphism
$\bar \varphi^0_{k}$
of coefficients rings.
Hence
$\bar \varphi^0_{\Pro^{\infty}}$
is an isomorphism as well.

Consider now
$A= \Gr(n)_+$. By Theorem
\ref{CohomologyOfGr}
its
$\mathrm{MGL}$-cohomology ring
is the ring of formal power series on
the Chern classes
of the tautological bundle
$\mathcal T_n$
over the coefficient ring
$\mathrm{MGL}^{\ast,\ast}(k)$.
The same holds for the
$\mathrm{BGL}$-cohomology ring.
As observed above, the map
$\varphi$
preserves the first Chern class,
thus it preserves all Chern classes.
Thus
$\bar \varphi^0_{\Gr(n)}$
is an isomorphism.

Now consider
$A=\Th(\mathcal T_n)$.
The morphism
$\varphi$
respects Thom classes
(see
(\ref{ThomBarClass})
and
(\ref{ThomClass})).
The vertical arrows in the commutative diagram
\[ \xymatrix@C=5em{
\overline{\mathrm{MGL}}^0((\Th(\mathcal{T}_n), \ast)) \ar[r]^-{\bar \varphi^0_{\Th(\mathcal{T}_n), \ast)}} &
\mathrm{BGL}^0(\Th(\mathcal{T}_n))    \\
\overline{\mathrm{MGL}}^0(\Gr(n))  \ar[r]^-{\bar \varphi^0_{\Gr(n)}}
\ar[u]^-{\overline{\mathrm{thom}}^\mathrm{MGL}} &  \mathrm{BGL}^0(\Gr(n))
\ar[u]_-{{\mathrm{thom}}^{\mathrm{BGL}}}}\]
are isomorphisms induced by the
the Thom isomorphism \ref{ThomIsomorphism}.
The map
$\bar \varphi^0_{\Gr(n)}$
is an isomorphism by the preceding case,
thus
$\bar \varphi^0_{Th(\mathcal T_n), \ast)}$
is an isomorphism.

\end{proof}

\subsection{The main result}
Let $k$ be a field and $S= \mathrm{Spec}(k)$.
By Theorem
\cite[Theorem 2.2.1]{UniThm}
and Example
\ref{OrientationsOfMGLandK}
there exists a unique monoid homomorphism
\begin{eqnarray}
\label{MainMOrphism2}
\varphi\colon \mathrm{MGL} \to \mathrm{BGL}
\end{eqnarray}
in
$\mathrm{SH}(S)$
such that $\varphi(th^\mathrm{MGL}) =th^K$.
It induces a natural transformation
\begin{eqnarray}
\label{CFhomPeriodicbgl}
\bar \varphi_{A}\colon \overline{\mathrm{MGL}}^{\ast}({A})\colon = \mathrm{MGL}^{\ast}({A}) \otimes_{\mathrm{MGL}^{0}(k)}
\mathrm{BGL}^{0}(k) \to
\mathrm{BGL}^{\ast}({A}).
\end{eqnarray}

\begin{theorem}
\label{PreMainThm}
The homomorphism
\[ \bar \varphi_{A}\colon \mathrm{MGL}^{\ast}({A}) \otimes_{\mathrm{MGL}^{0}(k)}
\mathrm{BGL}^{0}(k) \to
\mathrm{BGL}^{\ast}({A}) \]
is an isomorphism for all small pointed motivic spaces.
\end{theorem}

\begin{proof}
In fact,
$(\mathrm{MGL}, th^{\mathrm{MGL}})$
is Quillen universal
by
\cite{Vez}
and
\cite{UniThm},
and weakly
$\mathrm{BGL}$-cellular by Theorem~\ref{Preliminary}. Theorem~\ref{AlmostlyMain}
completes the proof.
\end{proof}




\begin{remark}
\label{MorelHopkins}
There is an unpublished result due to Morel and Hopkins,
which states that there is a canonical isomorphism
of the form
$$
\mathrm{MGL}^{\ast,\ast}(X) \otimes_{\mathbb{L}} \mathbb{Z}[\beta, \beta^{-1}] \to
\mathrm{BGL}^{\ast,\ast}(X)
$$
where $\mathbb{L}$ denotes the Lazard ring carrying the universal formal group law.
If the canonical homomorphism
$\mathbb{L} \to \mathrm{MGL}^{0}(k)$
is an isomorphism, Theorem
\ref{PreMainThm} implies their result.
\end{remark}

Let $X$ be a smooth $k$-scheme and $Z\subseteq X$ a closed subset, with
open complement $U\subseteq X$.
Consider the small pointed motivic space
$X/U$
and take the quotients of both sides of the isomorphism
(\ref{CFhomPeriodicbgl})
modulo the principal ideal
generated by the element
$1\otimes(\beta +1)$.
Corollary
\ref{LaurantPolinomial}
then implies that the natural transformation
\begin{eqnarray}
\label{CFhomomorphism}
\bar {\bar \varphi}_{X/U}\colon  
\mathrm{MGL}^{\ast}(X/U) \otimes_{\mathrm{MGL}^{0}(k)}
\mathbb{Z} \to K^{TT}_{- \ast,Z}(X)
\end{eqnarray}
is an isomorphism,
where
$K^{TT}_{\ast,Z}(X)$
are the
Thomason-Trobaugh $K$-groups with supports.
This family of isomorphisms shows that the functor
\[ (X,U)\mapsto \mathrm{MGL}^\ast(X/U)\otimes_{\mathrm{MGL}^0(k) }\ZZ
\]
is a ring cohomology theory in the sense of
\cite{PSorcoh}.
This implies the first part of our main result.

\begin{theorem}[Main Theorem]
\label{MainThm}
Let $X\in \Sm/k$ and $Z\subseteq X$ be a closed subset.
\begin{itemize}
\item
The family of isomorphisms
\begin{equation}
\label{CFhomomorphism2}
\bar {\bar \varphi}_{X/(X-Z)}\colon
\mathrm{MGL}^{\ast}(X/(X-Z)) \otimes_{\mathrm{MGL}^{0}(k)}
\mathbb{Z} \to K^{TT}_{- \ast,Z}(X)
\end{equation}
form an isomorphism
$\bar {\bar \varphi}$
of ring cohomology theories on
$\SmOp/k$.
\item
The natural isomorphism
$\bar {\bar \varphi}$ respects orientations provided that
$\mathrm{MGL}^{\ast}$
and
$K^{TT}_{- \ast}$ are considered
as oriented cohomology theories
in the sense of
\cite{PSorcoh}
with orientations given
by the Thom class
$th^{\mathrm{MGL}} \otimes 1$
from
\ref{OrientationsOfMGLandK}
and
the Chern structure
$L/X \mapsto [\mathcal O]-[L^{-1}]$.
In particular, the composition
\[ \xymatrix@C=5em{
  \mathrm{MGL}^{0}(k) \ar[r]^-{a\mapsto a\tensor 1} &
  \mathrm{MGL}^0(k) \tensor \ZZ \ar[r]^-{b\tensor c \mapsto \varphi(b)\cdot c} &  \mathbb Z} \] 
sends the class
$[X] \in \mathrm{MGL}^{0}(X)$
of a smooth projective $k$-variety
$X$ to
the Euler
characteristic
$\chi(X, \mathcal O_X)$
of the structure sheaf
$\mathcal O_X$.
\end{itemize}
\end{theorem}

\begin{proof}
The first part is already proven. To prove the second part, consider
the orientations
$th^{\mathrm{MGL}}$
and
$th^K$
from
\ref{OrientationsOfMGLandK}.
Note that
by the very definition of
$\varphi$
it sends
$th^{\mathrm{MGL}}$
to
$th^K$.
Thus it respects the Chern structures on
$\mathrm{MGL}^{\ast}$
and
$\mathrm{BGL}^{\ast}$
described in Section
\ref{OrientedSpectraAndTheory}.

The quotient map
$\mathrm{BGL}^{\ast} \to K^{TT}_{- \ast}$
takes the Bott element
$\beta$ to $(-1)$.
Thus it takes the Chern structure on
$\mathrm{BGL}^{\ast}$
to the Chern structure on
$K^{TT}_{- \ast}$
given by
$L/X \mapsto [\mathcal O]-[L^{-1}] \in K_0(X)$.
This shows that
$\bar {\bar \varphi}\colon \mathrm{MGL}^{\ast}(-) \otimes_{\mathrm{MGL}^{0}(k)}
\mathbb{Z}^{\ast} \to K^{TT}_{- \ast}$
respects the orientations described in the Theorem
\ref{MainThm}.

Let
$f \mapsto f_{\mathrm{MGL}}$
resp.~$f \mapsto f_K$
be the integrations on
$\mathrm{MGL}^{\ast}$
resp.~$K^{TT}_{- *}$
given by these Chern structures via Theorem
\cite[Thm.~4.1.4]{PSpush}.
By Theorem
\cite[Thm.~1.1.10]{PSrr}
the composition
$\mathrm{MGL}^\ast \to \mathrm{BGL}^\ast \to K^{TT}_{- \ast}$
respects the integrations on
$\mathrm{MGL}^{\ast}$
and
$K^{TT}_{- \ast}$
since it preserves the Chern structures.
In particular, given a smooth projective
$S$-scheme $f\colon X \to \Spec(k)$,
the diagram
\[
\xymatrix{
\mathrm{MGL}^{0}(X) \otimes_{\mathrm{MGL}^{0}(k)}
\mathbb{Z} \ar[r]^-{\bar {\bar \varphi}} \ar[d]_-{f_\mathrm{MGL}} &
  K^{TT}_0(X) \ar[d]^-{f_K}    \\
\mathrm{MGL}^{0}(X) \otimes_{\mathrm{MGL}^{0}(k)}
\mathbb{Z}  \ar[r]^-{\bar {\bar \varphi}} &  K^{TT}_0(k)} \]
commutes
where
$f_{\mathrm{MGL}}$
and
$f_K$
are the push-forward maps for
$\mathrm{MGL}^{\ast}$
and
$K^{TT}_{- \ast}$
respectively.
The integration
$f \mapsto f_K$
on
$K^{TT}_{- \ast}$
respecting the Chern structure
$L \mapsto [\mathcal O]-[L^{-1}]$
coincides with the one given by
the higher direct images by Theorem
\cite[Thm.~1.1.11]{PSrr}.
The last one sends the class
$[V] \in K_0(X)$
of a vector bundle $V$ over a smooth projective
variety $X$ to the Euler characteristic
$\chi(X, \mathcal V)$
of the sheaf $\mathcal{V}$ of sections of
$V$.

Recall that for an oriented cohomology theory
$A$ with a Chern structure
$L \mapsto c(L)$
and for a smooth projective variety
$f\colon X \to \Spec(k)$
its class
$[X]_A \in A(\Spec(k))$
is defined as
$f_A(1)$.
Here
$f_A\colon A(X) \to A(\Spec(k))$
is the push-forward morphism respecting the Chern structure
(see
\cite[Thm.~4.1.4]{PSpush}).
The notation $f_A$ is misleading, since $f_A$ depends on the Chern structure
as well.
Taking the element
$1\in \mathrm{MGL}^{0,0}(X)$
and using the commutativity of the very last
diagram  we see that
\[\bar {\bar \varphi} ([X]_{\mathrm{MGL}} \otimes 1)= \chi(X, \mathcal{O}_X).\]
Whence the Theorem.
\end{proof}

\section{Acknowledgements}
The work was supported by 
the SFB 701 at the Universit\"at Bielefeld, the RTN-Network HPRN-CT-2002-00287,
the grants RFFI 03-01-00633a and INTAS-05-1000008-8118, and the Fields Institute for
Research in Mathematical Science.

\end{document}